\documentclass[12pt]{article}

\usepackage{amsmath,amsthm,amsfonts,amssymb,latexsym,amscd,bbm,mathtools}
\usepackage{color}

\begin{document}

\newtheorem{theorem}{Theorem}
\newtheorem{corollary}[theorem]{Corollary}
\newtheorem{lemma}[theorem]{Lemma}
\newtheorem{proposition}[theorem]{Proposition}
\newtheorem{conjecture}[theorem]{Conjecture}
\newtheorem{commento}[theorem]{Comment}
\newtheorem{definition}[theorem]{Definition}
\newtheorem{problem}[theorem]{Problem}
\newtheorem{remark}[theorem]{Remark}
\newtheorem{remarks}[theorem]{Remarks}
\newtheorem{example}[theorem]{Example}

\newcommand{\Nb}{{\mathbb{N}}}
\newcommand{\Rb}{{\mathbb{R}}}
\newcommand{\Tb}{{\mathbb{T}}}
\newcommand{\Zb}{{\mathbb{Z}}}
\newcommand{\Cb}{{\mathbb{C}}}
\newcommand{\Qb}{{\mathbb{Q}}}

\newcommand{\Af}{\mathfrak A}
\newcommand{\Bf}{\mathfrak B}
\newcommand{\Ef}{\mathfrak E}
\newcommand{\Ff}{\mathfrak F}
\newcommand{\Gf}{\mathfrak G}
\newcommand{\Hf}{\mathfrak H}
\newcommand{\Kf}{\mathfrak K}
\newcommand{\Lf}{\mathfrak L}
\newcommand{\Mf}{\mathfrak M}
\newcommand{\Rf}{\mathfrak R}

\newcommand{\x}{\mathfrak x}

\def\A{{\mathcal A}}
\def\B{{\mathcal B}}
\def\C{{\mathcal C}}
\def\D{{\mathcal D}}
\def\F{{\mathcal F}}
\def\G{{\mathcal G}}
\def\H{{\mathcal H}}
\def\J{{\mathcal J}}
\def\K{{\mathcal K}}
\def\LL{{\mathcal L}}
\def\N{{\mathcal N}}
\def\M{{\mathcal M}}
\def\N{{\mathcal N}}
\def\OO{{\mathcal O}}
\def\P{{\mathcal P}}
\def\Q{{\mathcal Q}}
\def\SS{{\mathcal S}}
\def\T{{\mathcal T}}
\def\U{{\mathcal U}}
\def\W{{\mathcal W}}

\def\ext{\operatorname{Ext}}
\def\span{\operatorname{span}}
\def\clsp{\overline{\operatorname{span}}}
\def\Ad{\operatorname{Ad}}
\def\ad{\operatorname{Ad}}
\def\tr{\operatorname{tr}}
\def\id{\operatorname{id}}
\def\en{\operatorname{End}}
\def\aut{\operatorname{Aut}}
\def\out{\operatorname{Out}}
\def\hom{\operatorname{Hom}}
\def\im{\operatorname{im}}
\def\sa{\operatorname{sa}}

\def\la{\langle}
\def\ra{\rangle}
\def\rh{\rightharpoonup}

\title{On Cohomology for Product Systems}

\author{Jeong Hee Hong\footnote{This work was supported by National Research Foundation of Korea
Grant funded by the Korean Government (KRF--2008--313-C00039).}, 
Mi Jung Son$^{\ast}$,   
Wojciech Szyma{\'n}ski\footnote{This work was partially supported by  the FNU Project Grant 
`Operator algebras, dynamical systems and quantum information theory' (2013--2015).}}

\date{{\small 27 July 2014}}

\maketitle

\renewcommand{\sectionmark}[1]{}

\vspace{7mm}
\begin{abstract}
A cohomology for product systems of Hilbert bimodules is defined via the $\ext$ functor. For the class 
of product systems corresponding to irreversible algebraic dynamics, relevant resolutions are found explicitly and 
it is shown how the underlying product system can be twisted by the $2$-cocycles. In particular, this 
process gives rise to cohomological deformations of the $C^*$-algebras associated with the product system. 
Concrete examples of deformations of the Cuntz's algebra ${\mathcal Q}_\Nb$ arising this way are 
investigated and we show they are simple and purely infinite. 
\end{abstract}

\vfill\noindent {\bf MSC 2010}: 46L08, 46L65, 18G10

\vspace{3mm}

\newpage

\section{Introduction}

Applications of cohomology to deformations of $C^*$-algebras and von Neumann algebras have been studied 
for decades, and yet they remain an active area of research in this field. Amongst more recent contributions, 
we would like to mention the works of Buss and Exel on inverse semigroups, \cite{BE}, and of Kumjian, Pask 
and Sims on higher-rank graphs, \cite{KPS}.  Often deformation of the $C^*$-algebra is related to a 
cohomological perturbation of another underlying object. A typical example of such process comes from  
a {\em twisted} (semi)group action leading to the {\em twisted} crossed product. 

In the present paper, we introduce a cohomology theory for product systems of Hilbert bimodules over 
discrete semigroups, as defined by Fowler in \cite{F}. Interestingly, better understanding of twisting of 
semigroup actions was one of the motivations behind the very introduction of such product systems, \cite{FR}. 
In Section 3, we take the classical point of view, \cite{B,H}, and define cohomology groups of a product 
system $X$ via the $\ext$ functor applied to a suitable module $\Mf$ of a ring $\Rf$ naturally associated 
with $X$. First examples include  cohomologies of groups, graphs, and certain product systems  
arising from semigroup actions on abelian groups. 

In Section 4, we restrict attention to a certain class of product systems arising from irreversible algebraic dynamics, 
corresponding to actions of discrete semigroups $P$ on compact groups. 
For such product systems, we construct explicitly a free resolution of module $\Mf$ and thus obtain 
working formulae for cocycles and coboundaries. The construction of the 
resolution takes advantage of the fact that all fibers $X_p$ 
of these systems $X= \bigsqcup_{p\in P}X_{p}$
are free modules over the coefficient $C^*$-algebra $A$. To each $2$-cocyle $\xi$ 
we associate a twisted product system $X^\xi$. The twisting is obtained by perturbing multiplication between 
the fibers. Then each twisted product system $X^\xi$ gives rise to several $C^*$-algebras, including the 
Toeplitz algebra $\T(X^\xi)$ and the Cuntz-Pimsner algebra $\OO(X^\xi)$. These algebras may be considered 
twisted versions of the Toeplitz algebra $\T(X)$ and the Cuntz-Pimsner algebra $\OO(X)$, respectively, 
associated with the original product system $X$. 

In Section 5, we test this deformation procedure on the product system $X$ whose Cuntz-Pimsner algebra 
$\OO(X)$ coincides with Cunt's algebra $\Q_\Nb$ associated to the $ax+b$-semigroup over $\Nb$, see 
\cite{C,HLS1}.  We look at certain numerical $2$-cocycles $\xi$ and show that the corresponding twisted 
$C^*$-algebras $\OO(X^\xi)$ are purely infinite and simple. 

\medskip\noindent
{\bf Acknowledgements.} 
The first named author would like to thank faculty and staff of the Department of Mathematics and 
Computer Science at the University of Southern Denmark in Odense for their warm hospitality during her 
sabbatical stay there in 2013--2014. The third named author is grateful to Alex Kumjian for useful 
discussions of graph cohomology, and to Nicolai Stammeier for valuable discussions and information on 
irreversible algebraic dynamics and their product systems.  


\section{Preliminaries on product systems}\label{Prelim-1}

Let $A$ be a $C^*$-algebra and $X$ be a complex vector space with a right action of $A$.
Suppose there is an $A$-valued inner product $\langle \cdot , \cdot \rangle _A$ on $X$ which is 
conjugate linear in the first variable and  satisfies
\begin{description}
\item{(i)} $\langle\xi, \eta\rangle _A =\langle\eta , \xi\rangle_A^*$,
\item{(ii)} $\langle\xi, \eta\cdot a\rangle_A=\langle\xi, \eta\rangle_A\, a$,
\item{(iii)} $\langle\xi, \xi\rangle_A \geq 0 $, and $\langle\xi, \xi\rangle_A=0 \; \Longleftrightarrow\; \xi=0$,
\end{description}
for $\xi,\eta\in X$ and $a\in A$. Then $X$ becomes a right Hilbert $A$-module
when it is complete with respect to the norm given by
$\|\xi\|:=\|\langle\xi,\xi\rangle _A\|^{\frac{1}{2}}$ for $\xi\in X$.

A module map $T:X\rightarrow X $ is said to be adjointable if there is a map $T^*:X\rightarrow X$
such that 
$$
\langle T\xi,\zeta\rangle_A=\langle\xi,T^*\zeta\rangle_A
$$ 
for all $\xi,\eta\in X$. An
adjointable map is automatically norm-bounded, and the set $\LL(X)$ of all adjointable
operators on $X$ endowed with the operator norm
is a $C^*$-algebra. The rank-one operator $\theta_{\xi, \eta}$ defined on $X$ as
$$
\theta_{\xi,\eta}(\zeta)=\xi\langle \eta,\zeta\rangle_A \; \; \text{for} \; \xi, \eta, \zeta\in X,
$$
is adjointable and we have $\theta_{\xi,\eta}^*=\theta_{\eta,\xi}$. Then
$\K(X)=\clsp\{\theta_{\xi, \eta}\mid \xi, \eta\in X\}$ is the ideal of compact operators in $\LL(X)$. 

Suppose $X$ is a right Hilbert $A$-module. A $*$-homomorphism 
$\varphi:A\rightarrow\LL(X)$ induces a left action of $A$
on a $X$ by $a\xi:=\varphi(a)\xi$, for $a\in A$ and $\xi\in X$.
Then $X$ becomes a right Hilbert $A$--$A$-bimodule (or 
$C^*$-correspondence over $A$). The standard bimodule $_AA_A$ is
equipped with $\langle a, b\rangle _A=a^*b$, and the right and left actions are simply given
by right and left multiplication in $A$, respectively.

For right Hilbert $A$--$A$-bimodules $X$ and $Y$, the balanced tensor product $X\otimes_AY$ becomes
a right Hilbert $A$--$A$-bimodule with the right action from $Y$,  the left action implemented by the
homomorphism $A\ni a \mapsto\varphi(a)\otimes_A \id_Y\in\LL(X\otimes_A Y)$, and the $A$-valued
inner product given by 
$$
\langle\xi_1\otimes_A\eta_1 , \xi_2\otimes_A\eta_2\rangle_A=
\langle \eta_1, \langle\xi_1,\xi_2\rangle_A\cdot\eta_2\rangle_A,
$$ 
for $\xi_i\in X$ and $\eta_i\in Y$, $i=1,2$.

Let $P$ be a multiplicative semigroup with identity $e$, and let $A$ be a $C^*$-algebra.
For each $p\in P$ let $X_p$ be a complex vector space.
Then the disjoint union $X := \bigsqcup_{p\in P}X_{p}$ is a {\em product system} over $P$
if the following conditions hold:
\begin{description}
\item{(PS1)} For each $p\in P\setminus\{e\}$, $X_p$ is a right Hilbert $A$--$A$-bimodule.
\item{(PS2)} $X_e$ is the standard bimodule $_AA_A$.
\item{(PS3)} $X$ is a semigroup such that $\xi\eta\in X_{pq}$ for $\xi\in X_p$ and $\eta\in X_q$,
and for $p, q\in P\setminus\{e\}$, this product extends to an isomorphism
$F^{p, q} : X_p\otimes_A X_q\rightarrow X_{pq}$ of right Hilbert $A$--$A$-bimodules. 
If $p$ or $q$ equals $e$ then the corresponding product in $X$ is induced by the left 
or the right action of $A$, respectively.
\end{description}

\begin{remark}\rm
For $p\in P$, there are maps
$F^{p, e}:X_p\otimes_A X_e\rightarrow X_p$ and $F^{e, p}:X_e\otimes_A X_p\rightarrow X_p$
by multiplication, ie $F^{p, e}(\xi\otimes a)=\xi\, a$ and $F^{e, p}(a\otimes\xi)=a\, \xi$
for $a\in A$ and $\xi\in X_p$. Note that $F^{p,e}$ is always isomorphism. However, $F^{e,p}$ is
isomorphism only if $\overline{\varphi(A)X_p}=X_p$ or, in the terminology from \cite{F}, if $X_p$ is
essential.
\end{remark}

For each $p\in P$, we denote by $\langle\cdot,\cdot\rangle_p$ the $A$-valued
inner product on $X_p$  and by $\varphi_p$ the $*$-homomorphism from $A$ into $\LL(X_p)$.
Due to associativity of the multiplication on $X$, we have
$\varphi_{pq}(a)(\xi\eta)=(\varphi_p(a)\xi)\eta$ for all $\xi\in X_p$, $\eta\in X_q$, and $a\in A$.

 For each pair $p,q\in P\setminus\{e\}$,
the isomorphism $F^{p,q} : X_p\otimes_A X_q\rightarrow X_{pq}$
allows us to define a $*$-homomorphism $i_p^{pq}:\LL(X_{p})\to\LL(X_{pq})$ by
$i_p^{pq}(S)=F^{p,q}(S\otimes_A I_q)(F^{p,q})^*$ for $S\in\LL(X_p)$.
In the case $r\neq pq$ we define $i_p^r:\LL(X_p)\to\LL(X_r)$ to
be the zero map $i_p^r(S)=0$ for all $S\in\LL(X_p)$. Further, we let $i_e^q=\varphi_q$.

\smallskip
Let $X=\sqcup_{p\in P}X_p$  be a product system over $P$ of right Hilbert $A$--$A$-bimodules. 
A map $\psi$ from $X$ to a  $C^*$-algebra $C$ is a {\em Toeplitz
representation} of $X$ if the following conditions hold:
\begin{description}
\item{(T1)} for each $p\in P\setminus\{e\}$, $\psi_p:=\psi\vert_{X_p}$ is linear,
\item{(T2)} $\psi_e:A\longrightarrow C$ is a $*$-homomorphism,
\item{(T3)} $\psi_p(\xi)\psi_q(\eta)=\psi_{pq}(\xi\eta)\; \; $ for $ \; \xi\in X_p$, $\eta\in X_q$, $p,q\in P$,
\item{(T4)} $\psi_p(\xi)^*\psi_p(\eta)=\psi_e(\langle\xi, \eta\rangle_p)$ for $\xi, \eta\in X_p$.
\end{description}
As shown in \cite{P}, for each $p\in P$ there exists a $*$-homomorphism
$\psi^{(p)} : \K(X_p) \longrightarrow C$ such that $\psi^{(p)}(\theta_{\xi,\eta})=
\psi_p(\xi)\psi_p(\eta)^*\, , \; \text{for}\; \xi,\eta\in X_p$. A Toeplitz representation $\psi$ is
{\it Cuntz-Pimsner covariant}, \cite{F}, if 
\begin{description}
\item{(CP)} $\psi^{(p)}(\varphi_p(a))=\psi_e(a)$ for $a\in\varphi_p^{-1}(\K(X_p))$  and all $p\in P$. 
\end{description}

The Toeplitz algebra $\T(X)$ associated to the product system $X$ was defined by Fowler as  the universal
$C^*$-algebra for Toeplitz representations, \cite{F}. The Cuntz-Pimsner algebra $\OO(X)$ is 
universal for the Cuntz-Pimsner covariant Toeplitz representations. A number of other related constructions 
exist in the literature, we do not discuss in here. However, we would like to mention co-universal 
algebras studied by Carlsen, Larsen, Sims and Vittadello in \cite{CLSV}, and reduced Cuntz-Pimsner 
algebras investigated by Kwa\'{s}niewski and Szyma\'{n}ski in \cite{KS}.


\section{A cohomology for product systems}\label{cohodefi}

Let $X$ be a product system of Hilbert bimodules over a semigroup $P$  and with the coefficient 
(unital) $C^*$-algebra $A$. Then the direct sum of $A-A$-bimodules 
\begin{equation}\label{r}
\Rf := \bigoplus_{p\in P} X_p
\end{equation}
becomes a ring graded over $P$ with the multiplication borrowed from $X$. We assume that there exists 
a unital $A$-bimodule map $\Psi:\Rf\to A$ such that 
\begin{equation}\label{psiidentity}
\Psi(xy) = \Psi(x\Psi(y)), 
\end{equation}
for all $x,y\in\Rf$. Then $A=X_e$ becomes a left $\Rf$-module, with the $\Rf$-action $\rh$ given 
by the composition of the multiplication in $\Rf$ with $\Psi$, i.e. 
\begin{equation}\label{action}
x\rh a := \Psi(xa), 
\end{equation}
for $x\in\Rf$, $a\in A$. We denote this module $\Mf$ and define the $n^{\rm th}$-cohomology group 
of the product system $X$ relative to $\Psi$ as 
\begin{equation}\label{hnext}
H_\Psi^n(X) := \ext_\Rf^n(\Mf,\Mf), 
\end{equation}
cf. \cite{H}, \cite{B}.

\begin{example}\label{cohgraphs}\rm
Let $E$ be a {\em finite} directed graph, with vertices $E^0$, edges $E^1$, and range and source 
mappings $r:E^1\to E^0$ and $s:E^1\to E^0$, respectively. Let $X_1$ be the standard Hilbert bimodule 
associated with $E$, \cite{K}, with the finite-dimensional coefficient algebra $A$ generated by vertex projections. 
Let $X$ be the product system over the additive semigroup $\Nb$ generated by $X_1$. 
For each $n\in\Nb$, $X_n$ is the $\Cb$-span of paths of length $n$ (paths of length zero being vertices). Multiplication in ring $\Rf$ is simply given by concatenation of directed paths. For a path $\mu$ we set 
$\Psi(\mu):=s(\mu)$. Then for a path $\mu$ and a vertex $v$ we have $\mu \rightharpoonup v = s(\mu)$ if 
$v=r(\mu)$, and $0$ otherwise. 
\end{example}

\begin{example}\label{groupcoho}\rm
Let $G$ be a countable group. We set $P=G$ and $X_g=\Cb g$ for all $g\in G$. 
Then  $X$ is a product system with the usual group algebra multiplication and the inner products 
$\la zg, wg \ra_g = \overline{z}w1$, for $g\in G$ and $z,w\in\Cb$. 
We have $\Rf=\Cb G$, the usual complex group algebra. In this case, $\Psi$ is the 
trivial representation of $\Cb G$ and $\Mf$ is the trivial module. 
\end{example}

\begin{example}\label{cohcuntz}\rm
Here we consider a product system studied in \cite{L} in connection with Exel's approach to semigroup 
crossed products via transfer operator,  \cite{E}, and in \cite{Y} and \cite{HLS1} in connection with Cuntz's 
algebra $\Q_\Nb$, \cite{C}. The product system $X$ is over the multiplicative semigroup $\Nb^\times$. 
The coefficient algebra $A$ is $C(\Tb)$, 
and each fiber $X_p$ is a free left $A$-module of rank one with a basis vector $\mathbbm{1}_p$. The right 
action of $A$ is determined by $\mathbbm{1}_p a=\alpha_p(a)\mathbbm{1}_p$, where $\alpha_p:A\to A$ 
is an endomorphism such that $\alpha_p(a)(z)=a(z^p)$ for $a\in A$ and $z\in\Tb$. The inner product in 
fiber $X_p$ is given by $\la a\mathbbm{1}_p, b\mathbbm{1}_p \ra_p = L_p(a^*b)$, where $L_p:A\to A$ is 
a transfer operator for $\alpha_p$ such that $L_p(a)(z)=\frac{1}{p}\sum_{w^p=z}a(w)$. Fibers are multiplied 
according to the rule $(a\mathbbm{1}_p)(b\mathbbm{1}_q)=a\alpha_p(b)\mathbbm{1}_{pq}$.  

It is shown in 
\cite[Lemma 3.1]{HLS1} that the left action of $A$ on each fiber is by compact operators.  In fact, 
this product system belongs to the class of singly generated product systems of finite type, as introduced 
in \cite[Definition 3.5]{HLS2}. 

We set $\Psi(a\mathbbm{1}_p):=a$, for $p\in\Nb^\times$ and $a\in A$. Then the action of $\Rf$ on $\Mf$ is determined by $\mathbbm{1}_p\rightharpoonup a=\alpha_p(a)$, for $p\in\Nb^\times$ and $a\in\Mf$. 
\end{example}


\section{Irreversible algebraic dynamics}

In this section, we consider irreversible dynamical systems corresponding to injective homomorphisms 
of abelian groups. We follow the approach of Stammeier, \cite{St} (see also \cite{BLS}), building on the 
works of Exel and Vershik, \cite{EV}, Cuntz and Vershik, \cite{CV}, and Carlsen and Silvestrov, \cite{CaSil}. 
In particular, we use the description of the product systems naturally arising from such dynamics, 
due to Stammeier, \cite{Storal}. 

Let $G$ be a countable abelian group, and let $P$ be a semigroup with identity $e$.  
Let $\theta$ be an action of $P$ on $G$ by injective group homomorphisms.  
We denote by $A:=C^*(G)$ the group $C^*$-algebra of $G$. For each $p\in P$ let 
$X_p$ be a free left $A$-module of rank one with a basis element $\mathbbm{1}_p$. The right 
action of $A$ on $X_p$  is determined by $\mathbbm{1}_p a=\theta_p(a)\mathbbm{1}_p$, $a\in A$. 
The inner product in $X_p$ is defined as 
\begin{equation}\label{condexpsubgroup}
\la a\mathbbm{1}_p, b\mathbbm{1}_p\ra_p := \theta_p^{-1}E_p(a^*b), 
\end{equation}
for $p\in\Nb^\times$ and $a,b\in A$. Here $E_p:C^*(G)\to C^*(\theta_p(G))$ is the conditional 
expectation given by restriction. For $a=g$ and $b=h$ with $g,h\in G$, this yields 
\begin{equation}\label{Stinnerpr}
\la g\mathbbm{1}_p, h\mathbbm{1}_p\ra_p = \left\{ \hspace{-2mm}\begin{array}{ll} 
{\theta_p^{-1}(g^{-1}h)}    & \text{if } g^{-1}h\in\theta_p(G), \\ 0 & 
\text{otherwise}. \end{array} \right. 
\end{equation}
If index $[G:\theta_p(G)]$ is finite, then in the dual picture, with $\hat{\theta}_p$ 
acting on $C(\widehat{G})$, this inner product corresponds to the transfer operator given by averaging 
over the finitely many inverse image points, \cite{Storal}. Finally, fibers are multiplied according to the rule 
\begin{equation}\label{fibermulti}
(a\mathbbm{1}_p)(b\mathbbm{1}_q)=a\theta_p(b)\mathbbm{1}_{pq}.
\end{equation}  

In this case, ring $\Rf$ is the skew product $\Zb G\rtimes_\theta P$, with multiplication 
$$ (gp)(hq) = (g\theta_p(h))(pq), $$
$g,h\in G$, $p,q\in P$.  We take $\Psi(g{\mathbbm 1}_p):=g$, 
$g\in G$, $p\in P$. Then the action of $\Rf$ on $\Mf$ is given by 
$$ (g{\mathbbm 1}_p) \rh h = g\theta_p(h), $$
$g,h\in G$, $p\in P$.  Example \ref{cohcuntz} from Section 3 arises as a special case of this construction.   

\smallskip
Now, we describe an acyclic, free resolution of the $\Rf$-module $\Mf$. 
To this end, we define a complex of $\Rf$-modules and maps 
\begin{equation}\label{resolution}
\ldots \overset{\partial_2}{\longrightarrow} \Ff_2
\overset{\partial_1}{\longrightarrow} \Ff_1
\overset{\partial_0}{\longrightarrow} \Ff_0
\overset{\partial_{-1}}{\longrightarrow} \Mf 
\longrightarrow 0,  
\end{equation}
as follows. We let $\Ff_0$ be a free left $\Rf$-module of rank $1$ with a basis element $[\;\;]$. 
For $n\geq 1$, we let $\Ff_n$ be a free left $\Rf$-module with a basis 
\begin{equation}\label{basis} 
\{ [p_1,\ldots,p_n] \mid p_k\in P, k=1,\ldots,n\}. 
\end{equation} 
The maps $\partial_*$ are defined as $\Rf$-module homomorphisms such that
$$ \begin{aligned}
\partial_{-1}([\;\;]) & = 1, \\
\partial_0([p]) & = (\mathbbm{1}_p - 1)[\;\;], 
\end{aligned} $$
and for $n\geq 2$ we set 
$$ \begin{aligned}
\partial_{n-1}( [p_1,\ldots,p_n]) & = \mathbbm{1}_{p_1}[p_2,\ldots,p_n] \\
  & + \sum_{i=1}^{n-1}(-1)^i [p_1,\ldots,p_{i-1}, p_i p_{i+1},p_{i+2},\ldots,p_n] \\
  & + (-1)^n[p_1,\ldots,p_{n-1}].  
\end{aligned} $$
A routine calculation shows that 
$$ \partial_n\partial_{n+1}=0 $$ 
for all $n\geq -1$. 

To show that complex (\ref{resolution}) is acyclic, we construct splitting homotopies.  
That is, we define abelian group homomorphisms  $h_{-1}:\Mf\to\Ff_0$ and $h_n:\Ff_n\to\Ff_{n+1}$, 
$n\geq 0$, such that 
$$ \begin{aligned}
\partial_{-1}h_{-1} & = \id_{\Mf}, \\
\partial_n h_n + h_{n-1}\partial_{n-1} & = \id_{\Ff_n}, \;\;\; \text{for } n\geq 0. 
\end{aligned} $$
For example, we may take 
$$ \begin{aligned}
h_{-1}(a) & = a[\;\;], \\
h_0(a\mathbbm{1}_p[\;\;]) & = a[p], \\
h_n(a\mathbbm{1}_{p_0}[p_1,\ldots,p_n]) & = a[p_0,p_1,\ldots,p_n], \;\;\; n\geq 1, 
\end{aligned} $$
for $a\in C^*(G)$. 

\smallskip
Now, applying the $\hom_{\Rf}(*,\Mf)$ functor to chain complex (\ref{resolution}), with $\Mf$ deleted, 
we obtain the following complex of homogeneous cochains:
\begin{equation}\label{cochains}
0 \longrightarrow \hom_{\Rf}(\Ff_0,\Mf) 
\overset{\partial_0^*}{\longrightarrow} \hom_{\Rf}(\Ff_1,\Mf) 
\overset{\partial_1^*}{\longrightarrow} \ldots
\end{equation}
By definition, we have 
\begin{equation}\label{homogeneouscocycles}
H_\Psi^n(X) = \frac{\ker(\partial_n^*)}{\im(\partial_{n-1}^*)}. 
\end{equation}
Restricting in (\ref{cochains}) elements of  $\hom_{\Rf}(\Ff_n,\Mf)$ to the basis (\ref{basis}) 
of the free $\Rf$-module $\Ff_n$, we obtain the following complex of inhomogeneous cochains:
\begin{equation}\label{inhomogeneous}
0 \longrightarrow C^0(P,\Mf) 
\overset{\partial^0}{\longrightarrow} C^1(P,\Mf) 
\overset{\partial^1}{\longrightarrow} C^2(P,\Mf) 
\overset{\partial^2}{\longrightarrow} \ldots
\end{equation}
Here we denote:
$$ \begin{aligned}
C^0(P,\Mf) & = \Mf, \\
C^n(P,\Mf) & = \{ \xi: \bigtimes^n P \to \Mf \}, \;\;\; n\geq 1. 
\end{aligned} $$
The cochain maps are:
$$ \begin{aligned}
\partial^0(a)(p) & = \theta_p(a) - a, \\
\partial^n(\xi)(p_1,\ldots,p_{n+1}) & = \theta_{p_1}(\xi(p_2,\ldots,p_{n+1})) \\
  & + \sum_{i=1}^n (-1)^i \xi(p_1,\ldots,p_{i-1}, p_i p_{i+1}, p_{i+2},\ldots,p_{n+1}) \\
  & + (-1)^{n+1} \xi(p_1,\ldots,p_n),  
\end{aligned} $$
for $n\geq1$, $a\in \Mf$, $\xi\in C^n(P,\Mf)$, $p$ and $p_1,\ldots,p_{n+1}\in P$. We have 
\begin{equation}\label{inhomogeneouscocycles}
H_\Psi^n(X) \cong \frac{\ker(\partial^n)}{\im(\partial^{n-1})}. 
\end{equation}

Now, let $\xi:P\times P\to A_{\sa}$ be a normalized (i.e. $\xi(p,q)=0$ if $p=1$ or $q=1$) $2$-cocycle with 
self-adjoint values. We define a new product system $X^\xi$ over $P$ and with coefficients in $A$, as follows. 
For each $p\in P$, fiber $X^\xi_p$ coincides with $X_p$ (but we denote the generator by 
${\mathbbm 1}_p^\xi$ to avoid 
confusion). However, the multiplication between fibers is twisted by $\xi$ according to the rule 
\begin{equation}\label{twistedmulti}
(a{\mathbbm 1}_p^\xi)(b{\mathbbm 1}_q^\xi) := \exp(i\xi(p,q))a\theta_p(b){\mathbbm 1}_{pq}^\xi. 
\end{equation}
It is not difficult to verify that $X^\xi$ satisfies  axioms (PS1)--(PS3) of a product system, given in Section 2 above. 
Consequently, the corresponding Toeplitz and Cuntz-Pimsner algebras $\T(X^\xi)$ and $\OO(X^\xi)$, 
respectively, may be considered as $\xi$-twisted versions of $\T(X)$ and $\OO(X)$, respectively. 

\begin{proposition}\label{trivialcocycles}
Let $\xi,\eta$ be normalized, self-adjoint $2$-cocycles such that  $[\xi]=[\eta]$ in $H^2_\Psi(X)$. 
Then the corresponding twisted product systems $X^\xi$ and $X^\eta$ are isomorphic. 
\end{proposition}
\begin{proof}
By hypothesis, there is a $\psi:P\to \Mf$ such that $\xi-\eta=\partial^1(\psi)$. Replacing $\psi$ with 
$1/2(\psi+\psi^*)$ if necessary, we may assume that $\psi(p)$ is self-adjoint for all $p\in P$. Define a map 
$X^\xi\to X^\eta$ so that $a{\mathbbm 1}^\xi_p \mapsto \exp(i\psi(p))a{\mathbbm 1}^\eta_p$ for all 
$p\in P$, $a\in A$. One easily verifies that this map yields the required isomorphism between $X^\xi$ and $X^\eta$. 
\end{proof}


\section{Twisted ${\mathcal Q}_{\Nb}$} 

In this section, we apply the twisting procedure described in Section 4 to the product system $X$ discussed 
in Example \ref{cohcuntz} from Section 3. We begin by having a quick look at the $0-$, 
$1-$ and $2-$cohomology groups. The $0$-cohomology is clear. Indeed, it follows from 
(\ref{inhomogeneouscocycles}) that we simply have 
$$ \begin{aligned} 
H^0_\Psi(X) & = \{ a\in A \mid \alpha_p(a)=a, \forall p\in\Nb^\times \} \\
 & = \Cb 1.
\end{aligned} $$
Now, let $\xi(1)=0$ and for each {\em prime} $p\in\Nb^\times$ let $\xi(p)\in A$ be arbitrary. 
Let $1\neq q\in\Nb^\times$ have prime factorization $q=p_1\cdots p_m$, with 
$p_1\leq p_2\leq\ldots\leq p_m$. Proceeding by induction on $m$, define 
$\xi(q):=\alpha_{q/p_m}(\xi(p_m)) + \xi(q/p_m)$. Then $\xi:q\mapsto \xi(q)$, $q\in\Nb^\times$, is a $1$-cocycle. 
For $\xi$ to be a $1$-coboundary, there must exist a function $\psi\in C(\Tb)$ such that for 
all prime $p\in\Nb^\times$ and all $z\in\Tb$ we have 
$$ \psi(z) = \psi(z^p) - \xi(p)(z). $$
To construct such a $\psi$, fix a prime $p$ for a moment and define $\psi(z)$ for $z\in \Tb$ such that 
$z^{p^k}=1$, by induction on $k$, as follows. 
$$ \begin{aligned}
\psi(1) &: = 0, \\
\psi(z) & := \psi(z^p) - \xi(p)(z). 
\end{aligned} $$ 
In this way, $\psi$ is densely defined on $\Tb$ at all roots of unity. It follows that $\xi$ is a $1$-cocycle if and 
only if $\psi$ can be extended to a continuous function on the entire cirle $\Tb$. 

\smallskip
For a $2$-cocycle $\xi:\Nb^\times \times \Nb^\times \to A$, 
suppose $\psi:\Nb^\times \to A$ is such that $\xi=\partial^1(\psi)$. Then for any two primes 
$p,q$ we must have 
$$ \begin{aligned}
\psi(pq) & = \alpha_p(\psi(q)) + \psi(p) - \xi(p,q) \\
 & = \alpha_q(\psi(p)) + \psi(q) - \xi(q,p), 
\end{aligned} $$
and hence 
\begin{equation}\label{pq2coboundary}
(\psi(q)(z^p) - \psi(q)(z)) - (\psi(p)(z^q) - \psi(p)(z)) = \xi(p,q) - \xi(q,p) 
\end{equation}
for all $z\in\Tb$. Thus, for $\xi$ to give a non-zero element in $H^2_\Psi(X)$, it suffices to have  
$\xi(p,q)(1) \neq \xi(q,p)(1)$ for some primes $p$ and $q$. For a more specific example,  
let $\xi:\Nb^\times \times \Nb^\times \to \Cb 1$ be a map such that 
\begin{equation}\label{bicharacter}
\xi(mn,k)=\xi(m,k)+\xi(n,k) \;\;\; \text{and} \;\;\; \xi(m,nk)=\xi(m,n)+\xi(m,k). 
\end{equation}
Then $\xi$ is a $2$-cocycle. For example, given two distinct primes $p$ and $q$ and complex numbers 
$a,b,c,d$, we can set 
\begin{equation}\label{pqbicharacter}
\xi(mp^kq^l,np^rq^j) := (ak+bl)(cr+dj), 
\end{equation}
with $m,n$ relatively prime with both $p$ and $q$. By the above, if $ad\neq bc$ then $\xi$ is not 
a coboundary. 

\smallskip
Let $\xi:\Nb^\times \times \Nb^\times \to \Rb1$ be a $2$-cocycle defined in (\ref{pqbicharacter}), with 
$a,b,c,d$ real numbers. We denote by $u$ the standard unitary generator of $A=C(\Tb)$ and for $m\in\Nb^\times$ we denote by $s_m$ the canonical image of ${\mathbbm 1}_m^\xi$ in 
${\mathcal Q}_\Nb^\xi:=\OO(X^\xi)$. (Of course, $s_m$ depends also on $\xi$. We do not indicate this 
explicitly to lighten the notation.)  
Similarly to \cite{C} and \cite{HLS1}, each $s_m$ is an isometry 
and the following relations hold:
\begin{description}
\item{(QX1)} $s_m s_n= e^{i(ak+bl)(cr+dj)}s_{mn}$, 
\item{(QX2)} $s_mu^l=u^{ml}s_m$, for all $l\in\Zb$, 
\item{(QX3)} $\displaystyle{\sum_{k=0}^{m-1}u^ks_ms_m^*u^{-k}=1,}$
\end{description}
where $k,r$ are the numbers of $p$-factors of $m$ and $n$, respectively, and $l,j$ are the 
numbers of $q$-factors of $m$ and $n$, respectively. 

\begin{proposition}\label{simplicity}
$C^*$-algebra ${\mathcal Q}_\Nb^\xi$ is simple. 
\end{proposition}
A proof of simplicity of ${\mathcal Q}_\Nb^\xi$, claimed in Proposition \ref{simplicity} above, may be given 
as an application of \cite[Theorem 5.10]{KS}. This requires 
showing minimality and topological aperiodicity (in the sense of Definition 5.7 and Definition 5.3 of \cite{KS}, 
respectively) of the underlying product system $X^\xi$. Since both proofs are essentially the same as 
those from \cite[Section 6.5]{KS} (treating the case of untwisted ${\mathcal Q}_\Nb$), we omit the details.  

\smallskip
We want to investigate the structure of $C^*$-algebra $\Q_\Nb^\xi$ a little bit further. To this end, we note that 
$X^\xi$ is a {\em regular} product system (i.e. the left action $\varphi_m$ on each fiber $X_m^\xi$ is injective 
and by compacts, see \cite[Definition 3.1]{KS}) over an Ore semigroup $\Nb^\times$. Thus, it follows 
from a very general argument, \cite[Lemma 3.7]{KS}, that 
$$ \Q_\Nb^\xi=\clsp\{as_ms_n^*b \mid m,n\in\Nb^\times, a,b\in A\}. $$ 
Furthermore, 
$$ \F_\Nb^\xi := \clsp\{as_ms_m^*b \mid m\in\Nb^\times, a,b\in A\} $$
is a unital $*$-subalgebra of $\Q_\Nb^\xi$. Since the $\xi$-twist does not affect $\F_\Nb^\xi$, this algebra 
is unchanged by introduction of the cocycle. In fact, as shown by Cuntz in \cite[Section 3]{C}, it is a simple 
Bunce-Deddens algebra with a unique trace, \cite{BD,D}. 

In the present situation, since the enveloping group $\Qb_+^\times$ of $\Nb^\times$ is amenable, 
$\Q_\Nb^\xi=\OO(X^\xi)$ coincides with the reduced algebra $\OO((X^\xi)^r)$, \cite{Exame} and 
\cite{KS}, and with the co-universal algebra $\N\OO_{X^\xi}^r$, \cite{CLSV}. Thus, there exists 
a faithful conditional expectation $E:\Q_\Nb^\xi\to\F_\Nb^\xi$ onto $\F_\Nb^\xi$ such that for all 
$m,n\in\Nb^\times$, $a,b\in A$ we have 
$$ E(as_ms_n^*b)=0 \text{ if } m\neq n. $$

Let $\D_\Nb^\xi$ be the $C^*$-subalgebra of $\F_\Nb^\xi$ generted by all projections $u^ks_ms_m^*u^{-k}$, 
that is 
$$ \D_\Nb^\xi := \clsp\{u^ks_ms_m^*u^{-k} \mid m\in\Nb^\times, k\in\Zb\}. $$
Then, as in \cite[Section 3]{C}, $\D_\Nb^\xi$ is commutative and there exists a faithful conditional 
expectation $F:\F_\Nb^\xi\to\D_\Nb^\xi$ onto $\D_\Nb^\xi$ such that for all $m\in\Nb^\times$, $k,l\in\Zb$ 
we have 
$$ F(u^ks_ms_m^*u^{-l})=0 \text{ if } k\neq l. $$
The composition $G:=F\circ E$ yields a 
faithful conditional expectation from $\Q_\Nb^\xi$ onto $\D_\Nb^\xi$. We also recall from 
\cite[Lemma 3.2(a)]{C}, that for all $k\in\Zb$ and $m,n\in\Nb^\times$, we have 
\begin{equation}\label{cuntzequation}
u^ks_ms_m^*u^{-k} = \sum_{j=0}^{n-1}u^{k+jm}s_{mn}s_{mn}^*u^{-k-jm}. 
\end{equation}
One immediate consequence of this identity is that 
\begin{equation}\label{zeroequ}
s_r^*u^ts_r=0\; \text{ unless $t$ is divisible by $r$}. 
\end{equation}
Another one is the identity
\begin{equation}\label{lcm}
s_ms_m^*s_ns_n^* = s_{m\vee n}s_{m \vee n}^*, 
\end{equation}
where symbol $\vee$ denotes the least common multiple of two positive integers. 

\begin{lemma}\label{projectioninequality}
Let $k,l\in\Zb$ and $m,n\in\Nb^\times$. Then 
$$ u^ks_ms_m^*u^{-k} \leq u^ls_ns_n^*u^{-l} $$ 
if and only if both $m$ and $k-l$ are divisible by $n$. 
\end{lemma}
\begin{proof}
By (\ref{cuntzequation}), we have 
$$ \begin{aligned}
u^ks_ms_m^*u^{-k} & = \sum_{j=0}^{n-1}u^{k+jm}s_{mn}s_{mn}^*u^{-k-jm}, \\
u^ls_ns_n^*u^{-l} & = \sum_{j=0}^{m-1}u^{l+jn}s_{mn}s_{mn}^*u^{-l-jn}. 
\end{aligned} $$
Thus, $ u^ks_ms_m^*u^{-k} \leq u^ls_ns_n^*u^{-l} $ if and only if for each $j\in\{0,\ldots,n-1\}$ 
there is a $j'\in\{0,\ldots,m-1\}$ such that $k+jm=l+j'n$ in $\Zb_{mn}$. This clearly implies the claim. 
\end{proof}

Now, we will show that $C^*$-algebra $\Q_\Nb^\xi$ is purely infinite, 
as in the untwisted case, \cite[Theorem 3.4]{C}. 
Our proof immitates the classical argument of Cuntz, \cite{Co2}, employed also in \cite[Theorem 2.6]{CV}, 
and relies on the following technical lemma. 

\begin{lemma}\label{cuntzlemma}
Let $Q$ be a non-zero projection in $\D_\Nb^\xi$, and let $k_0,l_0\in\Zb$, $m_0,n_0\in\Nb^\times$ be 
such that either $k_0\neq l_0$ or $m_0\neq n_0$. Then there exist $k\in\Zb$ and $m\in\Nb^\times$ 
such that
\begin{description}
\item{(i)} $u^ks_ms_m^*u^{-k} \leq Q$, and 
\item{(ii)} $(u^ks_ms_m^*u^{-k})(u^{k_0}s_{m_0}s_{n_0}^*u^{-l_0})(u^ks_ms_m^*u^{-k}) =0$. 
\end{description}
\end{lemma}
\begin{proof}
By the definition of $\D_\Nb^\xi$, there exist $k'\in\Zb$ and $m'\in\Nb^\times$ such that 
$u^{k'}s_{m'}s_{m'}^*u^{-k'} \leq Q$. Thus, it suffices to work with $u^{k'}s_{m'}s_{m'}^*u^{-k'}$ 
instead of $Q$. 

If $u^{k'}s_{m'}s_{m'}^*u^{-k'}$ is not a subprojection 
of either $u^{k_0}s_{m_0}s_{m_0}^*u^{-k_0}$ or $u^{l_0}s_{n_0}s_{n_0}^*u^{-l_0}$, then 
to have (i) and (ii) satisfied it suffices 
to take $k\in\Zb$ and $m\in\Nb^\times$ such that either $u^ks_ms_m^*u^{-k} \leq 
u^{k'}s_{m'}s_{m'}^*u^{-k'}(1 - u^{k_0}s_{m_0}s_{m_0}^*u^{-k_0})$ or  $u^ks_ms_m^*u^{-k} 
\leq u^{k'}s_{m'}s_{m'}^*u^{-k'}(1 - u^{l_0}s_{n_0}s_{n_0}^*u^{-l_0})$, respectively. 

Now, we may assume that $u^{k'}s_{m'}s_{m'}^*u^{-k'}$ is a subprojection 
of both $u^{k_0}s_{m_0}s_{m_0}^*u^{-k_0}$ and $u^{l_0}s_{n_0}s_{n_0}^*u^{-l_0}$. 
Thus, by virtue of Lemma \ref{projectioninequality}, both $m'$ and $k'-k_0$ are divisible by $m_0$, 
while both $m'$ and $k'-l_0$ are divisible by $n_0$. Hence
$$ \begin{aligned}
 (u^{k'}s_{m'}s_{m'}^*u^{-k'}) & (u^{k_0}s_{m_0}s_{n_0}^*u^{-l_0})(u^{k'}s_{m'}s_{m'}^*u^{-k'}) \\
 &  = u^{k'}s_{m'}s_{m'}^*s_{m_0}u^{(k_0-k')/m_0 - (l_0-k')/n_0}s_{n_0}^*s_{m'}s_{m'}^*u^{-k'}
\end{aligned} $$
is a partial isometry with the domain projection 
$$ g=(u^{k'}s_{m'}s_{m'}^*u^{-k'}) u^{l_0-n_0(k_0-k')/m_0}s_{n_0(m'/m_0 \vee m'/n_0)}
   s_{n_0(m'/m_0 \vee m'/n_0)}^*u^{-(l_0-n_0(k_0-k')/m_0)} $$
and the range projection 
$$ f=(u^{k'}s_{m'}s_{m'}^*u^{-k'})u^{k_0-m_0(l_0-k')/n_0}s_{m_0(m'/m_0 \vee m'/n_0)}
   s_{m_0(m'/m_0 \vee m'/n_0)}^*u^{-(k_0-m_0(l_0-k')/n_0)}. $$
Clearly, both $g$ and $f$ are subprojections of $u^{k'}s_{m'}s_{m'}^*u^{-k'}$. If 
either $g\neq u^{k'}s_{m'}s_{m'}^*u^{-k'}$ or $f\neq u^{k'}s_{m'}s_{m'}^*u^{-k'}$ then we can 
argue as above. So suppose that both $g=u^{k'}s_{m'}s_{m'}^*u^{-k'}$ and $f=
u^{k'}s_{m'}s_{m'}^*u^{-k'}$. Then by Lemma \ref{projectioninequality}, $m'$ is divisible by both 
$n_0(m'/m_0 \vee m'/n_0)$ and $m_0(m'/m_0 \vee m'/n_0)$. This can only happen if $m_0=n_0$. 

Now, since $m_0=n_0$, $0\neq k_0-l_0$ is divisible by $m_0$. 
If we take $r\in\Zb$ relatively prime with $k_0-l_0$, then 
$$ (u^{k'}s_rs_r^*u^{-k'}) (u^{k_0}s_{m_0}s_{m_0}^*u^{-l_0})(u^{k'}s_rs_r^*u^{-k'}) = 
   u^{k'}s_rs_r^*u^{k_0-l_0}s_rs_r^*s_{m_0}s_{m_0}^*u^{-k'}=0 $$
by (\ref{zeroequ}). Thus, in this case, it suffices to put $k=k'$ and $m=r \vee m'$. 
\end{proof}

\begin{theorem}\label{purelyinfinite}
$C^*$-algebra ${\mathcal Q}_\Nb^\xi$ is purely infinite. 
\end{theorem}
\begin{proof}
Let $0\neq x\in\Q_\Nb^\xi$. Since $\Q_\Nb^\xi$ is simple, to show it is purely infinite as well we must find 
elements $T,R$ such that $TxR$ is invertible, \cite{Cpi}. 
We have $0\neq G(xx^*)\geq 0$. Thus there exists a projection $Q\in\D_\Nb^\xi$ 
such that $G(xx^*)$ is invertible in $Q\D_\Nb^\xi$. So let $d$ be a positive element of $\D_\Nb^\xi$ such that 
$G(dxx^*d)=d^2G(xx^*)=Q$. 

Now, take a small $\epsilon>0$. There exists a finite collection $m_j,n_j\in\Nb^\times$, $k_j,l_j\in\Zb$, 
$\lambda_j\in\Cb$ such that 
$$ ||dxx^*d - \sum_j \lambda_j u^{k_j}s_{m_j}s_{n_j}^* u^{-l_j}|| < \epsilon. $$
Applying conditional expectation $G$ we get 
$$ ||Q - \sum_{j:\; m_j=n_j, k_j=l_j} \lambda_j u^{k_j}s_{m_j}s_{m_j}^* u^{-k_j}|| < \epsilon. $$
Combining the two preceding inequalities, we see that 
\begin{equation}\label{inequality}
||dxx^*d - Q - \sum_{j:\;m_j\neq n_j\; {\text or }\;k_j\neq l_j}\lambda_j u^{k_j}s_{m_j}s_{n_j}^* u^{-l_j}|| 
< 2\epsilon.
\end{equation}
Now, applying repeatedly Lemma \ref{cuntzlemma}, we find a $k\in\Zb$ and an $m\in\Nb^\times$ such that 
$u^ks_ms_m^*u^{-k} \leq Q$ and $(u^ks_ms_m^*u^{-k})(u^{k_j}s_{m_j}s_{n_j}^*u^{-l_j})
(u^ks_ms_m^*u^{-k})=0$ for all $j$ with $m_j\neq n_j$ or $k_j\neq l_j$. Thus inequality 
(\ref{inequality}) yields 
$$ ||(u^ks_ms_m^*u^{-k})dxx^*d(u^ks_ms_m^*u^{-k}) - u^ks_ms_m^*u^{-k} || <2\epsilon. $$
Setting $T:=s_m^*u^{-k}d$ and $R:=x^*du^ks_m$ we have 
$$ || TxR - 1 || <2\epsilon, $$
and $TxR$ is invertible if $\epsilon \leq 1/2$. This proves  that $\Q_\Nb^\xi$ is purely infinite. 
\end{proof}


\vspace{5mm}\noindent
Jeong Hee Hong \\
Department of Data Information \\
Korea Maritime and Ocean University \\
Busan 606--804, South Korea \\
E-mail: hongjh@kmou.ac.kr \\

\smallskip\noindent
Mi Jung Son \\
Department of Data Information \\
Korea Maritime and Ocean University \\
Busan 606--804, South Korea \\
E-mail: mjson@kmou.ac.kr \\

\smallskip\noindent
Wojciech Szyma{\'n}ski\\
Department of Mathematics and Computer Science \\
The University of Southern Denmark \\
Campusvej 55, DK-5230 Odense M, Denmark \\
E-mail: szymanski@imada.sdu.dk


\begin{thebibliography}{99} 

\bibitem{B} S. Balcerzyk, 
{\em Wst\c{e}p do algebry homologicznej}, 
Biblioteka Matematyczna {\bf 34}, Pa\'{n}stwowe Wydawnictwo Naukowe, Warszawa, 1972. 

\bibitem{BLS} N. Brownlowe, N. S. Larsen and N. Stammeier, 
{\em On $C^*$-algebras associated to right $LCM$ semigroups}, 
arXv:1406.5725. 

\bibitem{BD} J. W. Bunce and J. A. Deddens, 
{\em A family of simple $C^*$-algebras related to weighted shift operators}, 
J. Funct. Anal. {\bf 19}  (1975), 13--24.

\bibitem{BE} A. Buss and R. Exel, 
{\em Twisted actions and regular Fell bundles over inverse semigroups}, 
Proc. Lond. Math. Soc. (3) {\bf 103}  (2011),  235--270.

\bibitem{CLSV} T. M. Carlsen, N. S. Larsen, A. Sims and S. Vittadello, 
{\em Co-universal algebras associated to product systems, and gauge-invariant uniqueness theorems},  
Proc. Lond. Math. Soc. (3) {\bf 103}  (2011),  563--600. 

\bibitem{CaSil} T. M. Carlsen and S. Silvestrov, 
{\em On the Exel crossed product of topological covering maps}, 
Acta Appl. Math. {\bf 108} (2009), 573--583.

\bibitem{Co2} J. Cuntz, 
{\em Simple $C^*$-algebras generated by isometries}, 
Commun. Math. Phys. {\bf 57}  (1977), 173--185.  

\bibitem{Cpi} J. Cuntz, 
{\em $K$-theory for certain $C^*$-algebras}, 
Ann. of Math. (2) {\bf 113}  (1981), 181--197. 

\bibitem{C} J. Cuntz,
{\em $C^*$-algebras associated with the $ax+b$-semigroup over $\Nb$},
in $K$-Theory and noncommutative geometry (Valladolid, 2006),
European Math. Soc., 2008, pp 201--215.

\bibitem{CV} J. Cuntz and A. Vershik, 
{\em $C^*$-algebras associated with endomorphisms and polymorphisms of compact abelian groups}, 
Commun. Math. Phys. {\bf 321} (2013), 157--179. 

\bibitem{D} K. R. Davidson, 
{\em $C^*$-algebras by example}. 
Fields Inst. Monographs, 6. Amer. Math. Soc., Providence, 1996.

\bibitem{Exame} R. Exel, 
\emph{Amenability for Fell bundles},
J. reine angew. Math. \textbf{492} (1997), 41--73.

\bibitem{E} R. Exel, 
{\em A new look at the crossed product of a $C^*$-algebra by an endomorphism}, 
Ergodic Theory \& Dynamical Systems {\bf 23} (2003), 1733--1750. 

\bibitem{EV} R. Exel and A. Vershik, 
{\em $C^*$-algebras of irreversible dynamical systems}, 
Canad. J. Math. {\bf 58}  (2006), 39--63. 

\bibitem{F} N. J. Fowler,
{\em Discrete product systems of Hilbert bimodules},
Pacific J. Math. {\bf 204} (2002), 335--375.

\bibitem{FR} N. J. Fowler and I. Raeburn, 
{\em Discrete product systems and twisted crossed products by semigroups}, 
J. Funct. Anal. {\bf 155}  (1998),  171--204.

\bibitem{HLS1} J. H. Hong, N. S. Larsen and W. Szyma{\'n}ski, 
The Cuntz algebra ${\mathcal Q}_\N$ and $C^*$-algebras of product systems, 
in `Progress in operator algebras, noncommutative geometry, and their applications',  97--109, 
Theta Ser. Adv. Math., 15, Theta, Bucharest, 2012.

\bibitem{HLS2} J. H. Hong, N. S. Larsen and W. Szyma\'{n}ski, 
{\em KMS states on Nica-Toeplitz algebras of product systems}, 
Internat. J. Math. {\bf 23} (2012) 1250123, pp. 1--38. 

\bibitem{H} S.-T. Hu, 
{\em Cohomology theory}, 
Markham Publishing Co., Chicago, 1968. 

\bibitem{K} T. Katsura, 
{\em A construction of $C^*$-algebras from $C^*$-correspondences}, 
in Advances in quantum dynamics (South Hadley, 2002), 173--182, Contemp. Math., 335, 
Amer. Math. Soc., Providence, 2003. 

\bibitem{KPS}  A. Kumjian, D. Pask and A. Sims, 
{\em Homology for higher-rank graphs and twisted $C^*$-algebras}, 
J. Funct. Anal.  {\bf 263}  (2012),  1539--1574.

\bibitem{KS} B. K. Kwa\'{s}niewski and W. Szyma\'{n}ski, 
{\em Topological aperiodicity for product systems over semigroups of Ore type}, 
arXiv:1312.7472. 

\bibitem{L} N. S. Larsen, 
{\em Crossed products  by abelian semigroups via transfer operators}, 
Ergodic Theory \& Dynamical Systems {\bf 30} (2010), 1147--1164. 

\bibitem{P} M. V. Pimsner,
{\em A class of $C^*$-algebras generalizing both Cuntz-Krieger algebras and crossed 
products by ${\mathbb Z}$},
Fields Inst. Commun. {\bf 12} (1997), 189--212.

\bibitem{St}  N. Stammeier, 
{\em $C^*$-algebras associated to certain semigroups of local homeomorphisms}, 
arXiv:1311.0793. 

\bibitem{Storal}   N. Stammeier, 
private communication.

\bibitem{Y} S. Yamashita,
{\em Cuntz's $ax+b$-semigroup $C^*$-algebra over $\Nb$ and product system $C^*$-algebras},
J. Ramanujan Math. Soc. {\bf 24} (2009), 299--322.

\end{thebibliography}
\end{document}